\documentclass{article}
\usepackage[a4paper,margin=1in]{geometry}
\usepackage{amsmath, amssymb, amsthm, hyperref,tikz,pgfplots}
\pgfplotsset{compat=1.18}
\usetikzlibrary{arrows.meta, decorations.markings, patterns}
\usepackage{authblk}
\DeclareMathOperator{\csch}{csch}

\newtheorem{theorem}{Theorem}[section]
\newtheorem{proposition}[theorem]{Proposition}
\newtheorem{lemma}{Lemma}[section]
\newtheorem{corollary}{Corollary}[section]

\theoremstyle{definition}
\newtheorem{definition}{Definition}[section]

\theoremstyle{remark}
\newtheorem{remark}{Remark}[section]
\author[1,2]{Mohammad Aqib\thanks{mohammadaqib@hri.res.in, aqibm449@gmail.com}}
\author[1,2]{Hemangi Madhusudan Shah\thanks{hemangimshah@hri.res.in}}

\affil[1]{Harish-Chandra Research Institute, Chhatnag Road, Jhunsi, Prayagraj-211019, India}
\affil[2]{Homi Bhabha National Institute, Training School Complex, Anushakti Nagar, Mumbai 400094, India}

\begin{document}
\title{Mixed Killing Vector Fields on Cigar Ricci-Bourguignon Solitons}
\date{}
\maketitle
\begin{abstract}
In this article, we study mixed Killing vector fields, defined by the condition 
$L_V L_V g = f\, L_V g$, on Cigar Ricci--Bourguignon solitons. 
While conformal vector fields are always mixed Killing, the converse fails in flat and open cylinders with base as manifold geometries, where the mixed Killing class is 
infinite-dimensional.  We establish a rigidity phenomenon of the Cigar Ricci--Bourguignon solitons: any complete steady almost gradient  Ricci--Bourguignon soliton on a surface with positive curvature is, up to homothety, Hamilton’s Cigar soliton. 
We then characterise complete mixed Killing fields, and affirm that locally any mixed Killing field is the sum of a rotationally Killing field and a mixed Killing radial field. Finally, we establish that the dimension of the vector space of complete mixed Killing fields of Cigar Ricci--Bourguignon solitons
is $5$. Moreover, we explicitly find its basis.
Our results show that Cigar Ricci--Bourguignon solitons depict 
completely different behaviour in contrast to Euclidean space.
Finally,  we also 
provide a complete description of the geodesic structure 
of  Cigar Ricci--Bourguignon solitons. 
\end{abstract}

%%%%%%%%%%%%%%%%%%%%%%%%%%%%%%%%%%%%%%%%%
\textbf {M. S. C. 2020:} 53B20, 53B25, 53C18, 53C35.\\
	%%%%%%%%%%%%%%%%%%%%%%%%%%%%%%%%%%%%%%%%%%%%%%%%%%%%%%%%%%%%%%%%%%%%
	
	%%%%%%%%%%%%%%%%%%%%%%%%%%%%%%%%%%%%%%%%%%%%%%%%%%%%%%%%%%%%%%%%%%%%

%%%%%%%%%%%%%%%%%%%%%%%%%%%%%%%%%%%%%%%%%%%%%%%%%%%%%%
\section{Introduction}

Geometric evolution equations have become central tools in differential 
geometry since Hamilton introduced the Ricci flow
\[
\frac{\partial g}{\partial t} = -2\operatorname{Ric}(g)
\]
\cite{Hamilton1988}. Self-similar solutions of Ricci flow, known as 
Ricci solitons satisfy
\[
\operatorname{Ric} + \frac{1}{2} L_{V} g  = \lambda g,
\]
where $V$ is a potential vector field and  $L_V$ 
denotes the Lie derivative with respect to $V$,
and play a fundamental role in singularity analysis and geometric 
classification.
A natural generalisation of Ricci flow  is the Ricci--Bourguignon flow
\[
\frac{\partial g}{\partial t}
= -2(\operatorname{Ric} - \rho R g),
\]
introduced in \cite{CCD17}. Its self-similar solutions (whenever they exist)
Ricci--Bourguignon solitons,  interpolate between the Ricci flow 
($\rho=0$) and the other curvature-driven evolutions. 
Allowing the soliton constant to vary leads to almost 
Ricci--Bourguignon solitons \cite{SDRBS}, 
which generalise almost Ricci solitons and arise naturally 
in warped product and Einstein-type geometries.
Mathematically, Ricci-Bourguignon  almost solitons are 
solutions of 
\[
\operatorname{Ric} + \frac{1}{2}\mathcal{L}_\xi g = (\lambda + \rho R)g,
\]
with a potential vector field $\xi$, soliton function $\lambda$, $\rho$ 
an arbitrary constant and $R$ scalar curvature.  \\

One of the prime examples of  Ricci solitons is the two-dimensional Hamilton's Cigar soliton.  
This is the unique complete, non-compact, steady gradient Ricci soliton 
with positive curvature \cite{Hamilton1988}.  The soliton data is   
\[
\left(\mathbb{R}^2,\;
g_{\mathrm{Cigar}} =\frac{dx^2+dy^2}{1+x^2+y^2},\;
  \xi=-2\Bigl(x\frac{\partial}{\partial x}
             +y\frac{\partial}{\partial y}\Bigr),\;
  \lambda=0\right).
\]
Here
\begin{equation}
  \xi := -2 \left(x\frac{\partial}{\partial x}
                     +y\frac{\partial}{\partial y}\right) = \nabla f, \\  \nonumber \;
 \mbox{where} \;  f(x,y) :=  - \log\bigl(1+x^2+y^2\bigr).                        
\end{equation}
It is conformally flat, asymptotically cylindrical, and has Gaussian curvature decaying exponentially at infinity. \\
The Cigar Ricci-Bourguignon  soliton was described in 
\cite{aqib2025characterizations}, and the soliton data is described as follows:
\[
\left(\mathbb{R}^2,\;
g_{\mathrm{Cigar-RB}}(t) =\frac{dx^2+dy^2}{E(t)+x^2+y^2},\;
  \xi=-2 (1-2 \rho) \Bigl(x\frac{\partial}{\partial x}
             +y\frac{\partial}{\partial y}\Bigr),\;
  \lambda=0\right).
\]
Here
\begin{equation}\label{xi-Cigar}
  \xi := -2 (1-2\rho) \left(x\frac{\partial}{\partial x}
                     +y\frac{\partial}{\partial y}\right) = \nabla f, \\  \nonumber \;
 \mbox{where} \;  f(x,y, t) :=  - (1-2 \rho)\log\bigl(E(t)+x^2+y^2\bigr)                
\end{equation}
and $E(t)=e^{4(1-2\rho)t}$.
It satisfies
\[
\operatorname{Ric} + \nabla^2 f - \rho R g = 0
\]
for all $t$. Thus $(\mathbb{R}^2, g_{\mathrm{Cigar-RB}}(t),\xi,\lambda\equiv0,\rho)$ is a steady 
gradient almost Ricci--Bourguignon soliton. 
See Section \ref{sec:Cigar-RB}  for more details on Cigar 
Ricci-Bourguignon almost solitons.\\

 Cigar Ricci-Bourguignon almost solitons are characterised by our Theorem  \ref{2D rigidity}: {\it Any complete surface 
admitting a steady almost gradient Ricci Bourguignon soliton with positive Gaussian curvature
 is isometric, up to scaling, to Hamilton's Cigar soliton}. The potential vector fields (the family of potential vector fields 
 depending on $\rho$) of Cigar gradient almost Ricci-Bourguignon solitons satisfy a remarkable property, an iterated Lie-derivative relation: its second Lie derivative of the metric is proportional to the first. See Proposition \ref{pro:mixed-killing-Cigar} for the derivation. This observation motivates the study of 
a class of vector fields, so-called \emph{mixed Killing} vector fields.\\

\indent
A vector field $V$ on a Riemannian manifold $(M,g)$ is called 
\emph{mixed Killing} if,
\[
L_V L_V g = f\, L_V g
\]
for some smooth function $f$. 
This notion, introduced in \cite{ghosh2025mixed}, generalizes 
2-Killing vector fields (where $f \equiv 0$, see \cite{Oprea2008}) 
and interpolates between Killing vector fields ($L_V g = 0$) and 
conformal vector fields   ($L_V g = \lambda g$, $\lambda$ is a function). 
Every conformal vector field is mixed Killing, whereas in flat 
and product geometries, the mixed Killing class is infinite-dimensional. See subsection 
\ref{sec:MK-Cigar} for further discussion.
It is therefore natural to ask whether curvature imposes rigidity on mixed Killing fields.\\\\
\indent
This paper aims to investigate mixed Killing vector fields 
on complete Cigar Ricci--Bourguignon solitons and to establish 
rigidity phenomena specific to this geometry. We establish a 
complete characterisation of mixed Killing vector fields. 
In particular, we show that every smooth angular mixed Killing 
field is a rotational Killing field. See Corollary \ref{mkc} for more details. The main result about complete mixed Killing vector fields on Cigar Ricci--Bourguignon solitons is that this space
$\mathfrak{MK}(g)$ has dimension $5$. See Theorem \ref{bmkv}. \\

We also provide an explicit description of the geodesics of 
Cigar Ricci--Bourguignon solitons and show that
all geodesics (except those starting at 
the tip) escape to infinity. This behaviour of the geodesics 
reflects the fact that any Cigar Ricci--Bourguignon soliton
is a complete non-compact surface. See  Section \ref{geocig} for more details.\\
 
 The article is divided into four sections.  In Section \ref{MKF}, 
we describe mixed Killing vector fields in flat and product manifolds.  
In Section \ref{sec:Cigar-RB}, we obtain rigidity of Cigar 
almost Ricci-Bourguignon solitons  (Theorem \ref{2D rigidity}). We also
show that the potential vector field of  Cigar almost Ricci-Bourguignon solitons
is a mixed Killing vector field.  In this section, we also obtain a basis and dimension for complete mixed Killing vector fields in Cigar almost Ricci-Bourguignon solitons, Theorem \ref{bmkv}.
In the final section, viz., Section \ref{geocig}, we analyse geodesics of Cigar almost Ricci-Bourguignon solitons.   

%%%%%%%%%%%%%%%%%%%%%%%%%%%%%%%%%%%%%%%%%%%%%%%%%%%%%%%%%%%%%%%%%%%%%%%%%

\section{Mixed Killing vector fields in flat and product manifolds}\label{MKF}
In this section,  we  first describe  mixed Killing vector fields on a 
Riemannian manifold.  This was introduced by Ghosh \cite{ghosh2025mixed}
as a natural generalisation of Killing and conformal vector fields. 
In subsection \ref{sec:MK-Cigar},  we show that mixed Killing vector fields in an open cylinder with base as manifolds, in particular flat manifolds, have infinite dimension. 

\begin{definition}[Mixed Killing Vector Field]\label{def:MK}
A vector field $V$ on a (semi) Riemannian manifold $(M,g)$ is said to be a 
{\it mixed Killing vector field}, if there exists a smooth function $f$ on $M$ 
such that
\begin{equation}\label{e1}
    L_V L_V g = f\,L_V g.
\end{equation}
We call $f$ the {\it mixed Killing factor} of $V$.
\end{definition}

Note that if $L_V g \equiv 0$ (i.e., $V$ is a Killing field), then $V$ is trivially 
mixed Killing with any function $f$. The interesting case is when $L_V g \neq 0$, 
in which case condition \eqref{e1} states that the second Lie derivative 
of the metric along $V$ is proportional to the first. This condition is weaker 
than being conformal but stronger than having no geometric constraint.\\

It is easy to see that conformal fields form a subclass of mixed Killing fields.

\begin{lemma}[Conformal fields are mixed Killing]\label{lem:conformal-MK-main}
Let $(M,g)$ be a (semi) Riemannian manifold and $V$ a conformal vector field 
with  $L_V g = 2\lambda\,g$  for some smooth function $\lambda$. Then on any open set where $L_V g \neq 0$, 
$V$ is a mixed Killing vector field with mixed Killing factor
\[
f = \frac{V(\lambda)}{\lambda} + 2\lambda.
\]
Consequently, (i) if $\lambda$ is constant (i.e., $V$ is homothetic), then $f \equiv 2\lambda$. And
(ii) if $V(\lambda) = 0$ (i.e. $\lambda$ is constant along flow lines of $V$), 
then $f = 2\lambda$.
\end{lemma}
\begin{proof}
From \eqref{e1} we have,
\[
L_V L_V g
 = L_V(2\lambda g)
 = 2(V\lambda)\,g + 2\lambda\,L_V g
 = 2(V\lambda)\,g + 4\lambda^2\,g
 = (2V\lambda + 4\lambda^2)\,g.
\]
On the other hand, $L_V g = 2\lambda g$, so on the set $\{\lambda\neq 0\}$ we can write
\[
L_V L_V g
 = \frac{2V\lambda + 4\lambda^2}{2\lambda}\,L_V g
 = \Bigl(\frac{V(\lambda)}{\lambda} + 2\lambda\Bigr)\,L_V g.
\]
Thus $V$ is mixed Killing in the sense of \eqref{e1}, with mixed Killing factor
$f$ as in the statement.
\end{proof}

\begin{corollary}\label{cor:homothetic-mixed-killing}
If $V$ is homothetic, i.e.\ $L_V g = 2c\,g$ for some constant $c\in\mathbb{R}$,
then $V$ is mixed Killing with \emph{constant} mixed Killing factor
$f = 2c$.
\end{corollary}

While conformal fields are always mixed Killing, the converse is false in general.\\

\noindent
{\bf Note:} It should be noted that the concept of a strictly nontrivial mixed Killing vector field $V$, that is, one 
which is {\it not} Killing makes sense on $M \setminus \{ L_{V} g 
\neq 0 \}$.

\subsection{Mixed Killing fields in  product geometries}\label{sec:MK-Cigar}
%%%%%%%%%%%%%%%%%%%%%%%%%%%%%%%%%%%%%%%%%%%%%%%%%%%%%%%%%%%%
  In this subsection, we now explore mixed Killing vector fields in the aforementioned product manifolds. A key observation is that in one-dimensional directions, the mixed Killing condition reduces to an ordinary differential equation.  

\begin{proposition}[One-dimensional reduction]\label{pro:1D-reduction}
Let $(M,g)$ be a Riemannian manifold and suppose that, in suitable
coordinates, a vector field has the form $V = v(x)\,\partial_x$,
and satisfies $L_V g = A(x)\,T,  L_V L_V g = B(x)\, T$,
for some symmetric $(0,2)$-tensor $T$. Then $V$ is a mixed Killing vector
field, $L_V L_V g = f\, L_V g$, if and only if  $f = \frac{B}{A}$,
at every point where $A \neq 0$ and $T$ is nonzero. In particular,
the mixed Killing condition reduces to an ordinary differential equation for $v$.
\end{proposition}

\begin{proof}
By assumption,
\begin{equation}\label{form}
L_V g = A(x)\,T, \qquad L_V L_V g = B(x)\,T.
\end{equation}
At any point where $A \neq 0$ and $T$ is nonzero, we may write
\[
L_V L_V g = \frac{B(x)}{A(x)}\,L_V g,
\]
thus the mixed Killing condition holds with $f(x) =  \frac{B(x)}{A(x)}$.\\
Conversely, if  $L_V L_V g = f\,L_V g$, then \eqref{form}  implies that
\[
B(x)\,T = f(x)\,A(x)\,T.
\]
Since $T$ is nonzero at the point in question, it follows that $f(x) = \frac{B(x)}{A(x)}$.
\end{proof}

Now we study a particular class of mixed Killing fields in manifolds of the type $I\times N$, in particular in flat manifolds  

\begin{proposition}[Mixed Killing fields on product manifolds]
\label{prop:product-mixed-killing}
Let $(M,g) = (I\times N,\,dx^2 + ds^2)$ be a Riemannian product, where
$I\subset\mathbb{R}$ is an open interval.
For any $v\in C^\infty(I)$, consider the vector field $V = v(x)\,\partial_x$.
Then
\[
L_V g = 2v'(x)\,dx^2, \qquad
L_V L_V g = \bigl(2v v'' + 4(v')^2\bigr)\,dx^2.
\]
Consequently,  (i) $V$ is mixed Killing on $\{v'\neq 0\}$ with factor
$f = v\,\frac{v''}{v'} + 2v'$  and (ii)  $V$ is conformally Killing if and only if $v'\equiv 0$
\end{proposition}

\begin{proof}
Clearly, $(L_V g)(\partial_x, \partial_x) = 2v'(x)$,
and all other components vanish. Applying $L_V$ again yields
\[
L_V L_V g = \bigl(2v v'' + 4(v')^2\bigr)\,dx^2,
\]
because the only nontrivial contribution comes from differentiating $v'$ along $V$.
On the set $\{v'\neq 0\}$, Proposition~\ref{pro:1D-reduction} applies with $T = dx^2$, giving
\[
f = \frac{2vv'' + 4(v')^2}{2v'} = v\frac{v''}{v'} + 2v'.
\]
If $v'\equiv 0$, then $L_V g=0$ and hence $V$ is Killing.
Since $L_V g$ has only a $dx^2$ component, $L_V g$ cannot be proportional to $g$ unless it vanishes identically; thus,$V$ is conformal if and only if $v'\equiv 0$.  
\end{proof}

\begin{remark}
In particular, on Euclidean space $(\mathbb{R}^2,dx^2+dy^2)$, any vector
field of the form $V=v(x)\partial_x$ with $v'\not\equiv 0$ is mixed Killing
but neither Killing nor conformal. Hence, in flat or product geometries of type $I \times N$,
the mixed Killing class is strictly much larger than the conformal class. But we will see in the section that Cigar almost Ricci-Bourguignon solitons behave in a dramatically different way. 
\end{remark}

\section{Characterization of 2-Killing vector fields on  Cigar almost Ricci-Bourguignon solitons}\label{sec:Cigar-RB}

In this section, we first characterise Cigar almost Ricci-Bourguignon solitons up to homothety.
It is interesting to see that the potential vector field of  Cigar almost Ricci-Bourguignon solitons  
is a mixed Killing vector field. Hence, it is natural to ask whether we can classify mixed Killing vector fields
on Cigar almost Ricci-Bourguignon solitons? We answer this affirmatively by categorising all the mixed Killing vector Fields on Cigar almost Ricci-Bourguignon solitons.  In fact, we find the 
basis of all mixed Killing vector fields and show that its
dimension is $5$.

\noindent
\subsection{Cigar almost  Ricci-Bourguignon solitons}\label{C-PC}
  Recall that a Cigar almost Ricci-Bourguignon  soliton  described in the introduction  has soliton data as:
\[
\left(\mathbb{R}^2,\;
g_{\mathrm{Cigar-RB}}(t) =\frac{dx^2+dy^2}{E(t)+x^2+y^2},\;
  \xi=-2 (1-2\rho) \Bigl(x\frac{\partial}{\partial x}
             +y\frac{\partial}{\partial y}\Bigr),\;
  \lambda=0, \;  \rho  \right),
\]
where  $E(t)=e^{4(1-2\rho)t}$  and the vector field
 $\xi = -2(1-2\rho)\left(x\frac{\partial}{\partial x}
                       +y\frac{\partial}{\partial y}\right) = \nabla f,$
where
\begin{equation*}
  f(x,y,t) := -(1-2\rho)\,\log D(t,x,y)
            = -(1-2\rho)\,\log\bigl(E(t)+x^2+y^2\bigr).
\end{equation*}

\vspace{0.1in}

\noindent
{\bf Geodesic polar coordinate transformation of Cigar Ricci--Bourguignon metric:}
Consider the transformation  $x = r \cos \theta, y = r \sin \theta$.
Then,
\begin{equation}\label{Cigar1}
g_{\mathrm{Cigar-RB}}(t) =   \frac{dr^2 + r^2 d\theta^2}{E(t) + r^2}, \qquad   E(t)=e^{4(1-2\rho)t}.
\end{equation}
If we consider the transformation in geodesic polar coordinates $(s,\theta)$ defined by
\[
r = \sqrt{E}\sinh s,
\]
then
\begin{equation}\label{iso1}
g_{\mathrm{Cigar-RB}}(t) =  ds^2 + \psi(s)^2\,d\theta^2, \qquad \psi(s)=\tanh s.
\end{equation}

\vspace{0.2in}
\noindent
{\bf Behaviour of Cigar metric near zeros of \boldmath$\tanh s$ \unboldmath:}
Note that as $s \rightarrow 0$,   $\tanh s \sim s$.
Hence, the metric looks like 
$g_{\mathrm{Cigar-RB}}(t) =  ds^2 + s^2 d\theta^2$
which is a flat metric.  
As $s \rightarrow \infty, \tanh s \rightarrow 1$, so it approaches a cylinder of radius $1$.
Observe that zeros of   $\tanh s$ are the same as zeros of  $\sinh s$, and it vanishes 
when $s = n \pi  i$. Near these points, the warped product metric behaves like 
$\tanh s \sim (s - s_{0})$, where $s_{0} =  n \pi  i,$ for some $n \in \mathbb{Z}$.
And as ${\tanh}' s_{0} =1$,  the metric becomes 
$g_{\mathrm{Cigar-RB}}(t) =  ds^2 + (s - s_{0})^2 d\theta^2$. 
This is just a flat polar metric on ${\mathbb R}^2$. So near each zero, the Cigar metric 
becomes locally Euclidean. And the Cigar has no conical singularity there; it is smooth and locally 
${\mathbb R}^2$.\\

From our discussion earlier, we see that the Cigar Ricci-Bourguignon metric \eqref{iso1}  satisfies the  equation 
\[
\operatorname{Ric} + \nabla^2 f - \rho R g = 0,
\]
for all $t$. Thus $(\mathbb{R}^2,g(t),\xi,\lambda\equiv 0,\rho)$ is a steady 
gradient almost Ricci--Bourguignon soliton.  Moreover:
\begin{enumerate}
\item At $t=0$ and $\rho=0$, we recover Hamilton's steady Cigar Ricci soliton.
\item For $\rho < 1/2$, the metric expands; for $\rho > 1/2$, it shrinks.
\item The soliton is self-similar under the one-parameter group of diffeomorphisms
\[
\varphi_t(x,y) = (a(t)x, a(t)y), \qquad a(t) = e^{-2(1-2\rho)t}.
\]
\end{enumerate}

\subsection{Rigidity for 2D almost Ricci--Bourguignon solitons}
We establish a classification result showing that the Cigar is the unique complete surface with positive curvature admitting a steady, almost gradient  Ricci-Bourguignon soliton structure.  First, we show that a Cigar Ricci-Bourguignon soliton has positive sectional curvature.

\begin{proposition}
The metric given by
\[
g = \frac{dr^2 + r^2 d\theta^2}{E(t) + r^2}, \qquad   E(t)=e^{4(1-2\rho)t}.
\]
has Gaussian curvature  
$$ K= \frac{2E}{(E+r^2)}.$$ Thus any Cigar Ricci 
-Bourguignon soliton has positive sectional curvature.
\end{proposition}

\begin{proof}
For a conformal metric
\[
g=u(r)(dr^2+r^2d\theta^2),
\]
the Gaussian curvature satisfies
\[
K
=
-\frac1{2u}
\Delta_{\mathbb R^2}(\log u).
\]

Therefore,
\[
K
=
\frac{2E}{(E+r^2)}.
\]
Thus, the Cigar Ricci-Bourguignon soliton has positive curvature.
\end{proof}

\begin{theorem}[2D rigidity]\label{2D rigidity}
Let $(\Sigma^2,g,f)$ be a complete steady almost gradient
Ricci-Bourguignon soliton with $\rho\neq \tfrac12$.  Suppose that $\nabla f$ has a zero.
If the Gauss curvature of $g$ is positive, then $(\Sigma^2,g)$ is isometric, up to homothety, to Hamilton's Cigar Ricci-Bourguignon soliton.
\end{theorem}

\begin{proof}
Let $(\Sigma^2,g,f)$ be a complete steady gradient Ricci--Bourguignon soliton
with $\rho\neq \tfrac12$ and positive Gauss curvature.
Thus
\begin{equation*}\label{RB}
\operatorname{Ric} + \nabla^2 f = \rho R g ,
\end{equation*}
where $\rho$ is a constant.
On a surface, $\operatorname{Ric} = \tfrac{R}{2}g$, hence  the above equation  reduces to
\begin{equation}\label{Hessian}
\nabla^2 f = \left(\rho-\tfrac12\right) R\, g .
\end{equation}
Therefore, $\nabla f$ is a gradient conformal vector field.
Since $(\Sigma^2,g)$ is complete with $K>0$, equation \eqref{Hessian} implies that the only possible critical point of  $f$ is either a minimum or a maximum.
By a standard result \cite{MR174022} on a complete surface admitting a nontrivial gradient 
conformal vector field with positive curvature, $(\Sigma^2,g)$ is rotationally symmetric about this point.\\
Consequently, in the neighbourhood of this point, $g$ can be written in polar coordinates as  
\[
g = dr^2 + h(r)^2\, d\theta^2, 
\]
where $h(0)=0$ and $h'(0)=1$. For such a metric, the Gauss curvature and scalar curvature are
\[
K = -\frac{h''}{h}, \qquad R = 2K = -\frac{2h''}{h}.
\]
From  (\ref{Hessian}), it follows that $f$ is also a radial function.  
The Hessian of $f$ is given by 
\[
(\nabla^2 f)_{rr} = f'', \qquad
(\nabla^2 f)_{\theta\theta} = h h' f'.
\]
Substituting into \eqref{Hessian} yields
\begin{equation}\label{system}
\begin{aligned}
f'' &= \left(\rho-\tfrac12\right) R, \\
\frac{h'}{h} f' &= \left(\rho-\tfrac12\right) R .
\end{aligned}
\end{equation}
Subtracting the two equations gives   $f'' - \frac{h'}{h}f' = 0$
which integrates to  $f' = a\, h$  for some constant $a$.
Substituting $f'=ah$ and $R=-2h''/h$ into \eqref{system}, we obtain
\[
(1-2\rho)h'' - a h h' = 0.
\]
Since $\rho\neq \tfrac12$, this can be written as
\[
h'' = k\, h h', \qquad k =\frac{a}{1-2\rho}.
\]
Set $u=h'(r)$. Then $h'' = u \frac{du}{dh}$, and the equation becomes
\[
u\frac{du}{dh} = k h u.
\]
Since $K>0$, we have $h>0$ and $h'>0$ for all $r>0$, so dividing by $u$ and
integrating yields
\[
u(h) = \frac{k}{2}h^2 + C.
\]
The smoothness conditions at $r=0$ imply $C=1$, hence
\[
h' = 1 + \frac{k}{2}h^2.
\]

Positive curvature implies $h''<0$ for $r>0$, and since $h,h'>0$, this forces
$k<0$. Writing $A =-\frac{k}{2}>0$, we obtain
\[
h' = 1 - A h^2.
\]
Separating variables and integrating gives
\[
\frac{1}{\sqrt A}\operatorname{arctanh}(\sqrt A\, h)= r,
\]
so that
\[
h(r)=\frac{1}{\sqrt A}\tanh(\sqrt A\, r).
\]

Therefore,
\[
g = dr^2 + \frac{1}{A}\tanh^2(\sqrt A\, r)\, d\theta^2.
\]
This metric is complete, has positive curvature, and coincides (up to scaling)
with Hamilton’s Cigar soliton. See  (\ref{iso1}) for more details.
Hence $(\Sigma^2,g)$ is a Cigar Ricci-Bourguignon soliton.
\end{proof}

Now we start our discussion of the mixed Killing vector field 
on Cigar Ricci-Bourguignon soliton.

\subsection{The conformal algebra of the Cigar}
To understand the mixed Killing vector fields of the Cigar soliton, we first show that its potential vector field is indeed
a mixed Killing vector field. Then we determine its conformal algebra. Note that the Lie algebra of conformal vector fields on a Riemannian manifold \(M\) is called the conformal algebra. 

\begin{proposition}\label{pro:mixed-killing-Cigar}
The potential vector field of  Cigar almost 
Ricci--Bourguignon  solitons  $(\mathbb{R}^2,g(t),\xi,\lambda\equiv 0,\rho)$ is a mixed Killing vector field.
\end{proposition}
\begin{proof}
The potential vector field is
\[
\xi
 = -2(1-2\rho)\left(x\frac{\partial}{\partial x}
                    +y\frac{\partial}{\partial y}\right).
\]
Set $r^2=x^2+y^2$ and $D(t,x,y):=E(t)+r^2$ where  $E(t)=e^{4(1-2\rho)t}.$
\[
L_\xi g = \psi\,g  \;\;\; \mbox{where} \;
\qquad 
\psi(x,y,t) = -\frac{4(1-2\rho)\,E(t)}{D(t,x,y)},
\]
and
\[
L_\xi L_\xi g =  \alpha \,L_\xi g
\]
holds with mixed Killing factor
\begin{equation}\label{eq:mixed-killing-factor-Cigar}
  \alpha(x,y,t)
  = -4(1-2\rho)\,\frac{E(t) - (x^2+y^2)}{E(t) + (x^2+y^2)}.
\end{equation}
In particular, $\xi$ is mixed Killing and $\alpha$ is not identically zero on
$\mathbb{R}^2$.
Note that $f \neq 0$ on $\mathbb{R}^2 \setminus \{r = \sqrt{E}\}$, 
confirming that $\xi$ is a genuinely non-trivial example of a mixed 
Killing field.
\end{proof}

\begin{theorem}[Conformal algebra of a Cigar Ricci -Bourguignon]\label{theorem:conf-4d}
Let $g$ be a Cigar-type Ricci--Bourguignon metric on $\mathbb{R}^2$,
\[
g = \phi(r)\,(dx^2 + dy^2),\qquad
\phi(r) := \frac{1}{E + r^2},\quad r^2 = x^2 + y^2, \;  E(t)=e^{4(1-2\rho)t}
\]
and let $\delta = dx^2 + dy^2$ be the Euclidean metric. Then:
\begin{enumerate}
\item[\textup{(i)}] A smooth vector field $V$ is conformal for $g$ if and only if it is conformal for $\delta$.

\item[\textup{(ii)}] The Lie algebra of complete conformal vector fields on $(\mathbb{R}^2,g)$ is $4$--dimensional and is spanned by
\[
\partial_x,\quad \partial_y,\quad -y\partial_x + x\partial_y,\quad x\partial_x + y\partial_y.
\]
\end{enumerate}
\end{theorem}

\begin{proof}
(i) Since $g=\phi\,\delta$ with $\phi>0$ smooth, we have, for any vector field $V$,
\[
L_V g
 = L_V(\phi\delta)
 = (V\phi)\,\delta + \phi\,L_V\delta.
\]
If $V$ is conformal for $\delta$, i.e.\ $L_V\delta = 2\lambda_\delta\,\delta$, then
\[
L_V g
 = (V\phi)\,\delta + 2\phi\lambda_\delta\,\delta
 = 2\Bigl(\lambda_\delta + \frac{V(\log\phi)}{2}\Bigr)\,\phi\delta
 = 2\lambda_g\,g,
\]
so $V$ is conformal for $g$ with
\[
\lambda_g = \lambda_\delta + \frac{V(\log\phi)}{2}. 
\]
\noindent
Conversely, if $V$ is conformal for $g$, $L_V g = 2\lambda_g g$, then
\[
(V\phi)\,\delta + \phi\,L_V\delta = 2\lambda_g\,\phi\,\delta.
\]
Dividing by $\phi$ gives
\[
L_V\delta = 2\Bigl(\lambda_g - \frac{V(\log\phi)}{2}\Bigr)\,\delta,
\]
so $V$ is conformal for $\delta$. This proves (i).

\smallskip
(ii) By (i), it suffices to determine the Lie algebra of complete conformal vector fields on $(\mathbb{R}^2,\delta)$.

Identify $(\mathbb{R}^2,\delta)$ with the complex plane $(\mathbb{C},|\cdot|^2)$ via $z=x+iy$. The global conformal diffeomorphisms of $\mathbb{C}$ are precisely the affine maps
\[
\Phi(z) = a z + b,\qquad a\in\mathbb{C}^\ast,\ b\in\mathbb{C},
\]
({\bf see \cite{MR503901}}. Thus
\[
\mathrm{Aut}(\mathbb{C}) = \{z\mapsto az+b\}
\]
is a $4$--dimensional real Lie group.
\noindent
A complete conformal vector field $V$ on $(\mathbb{R}^2,\delta)\cong\mathbb{C}$ generates a one-parameter group of conformal diffeomorphisms $\{\Phi_t\}_{t\in\mathbb{R}}\subset\mathrm{Aut}(\mathbb{C})$:
\[
\frac{d}{dt}\Phi_t(z) = V(\Phi_t(z)),\qquad \Phi_0=\mathrm{id}_{\mathbb{C}}.
\]
Since each $\Phi_t$ is affine, we can write $\Phi_t(z)=a(t)z + b(t)$, with $a(0)=1$, $b(0)=0$. Differentiating at $t=0$ yields
\[
V(z) = \left.\frac{d}{dt}\right|_{t=0}\Phi_t(z) = \alpha + \beta z
\]
for some $\alpha,\beta\in\mathbb{C}$. Conversely, any vector field of the form $V(z)=\alpha + \beta z$ generates a one-parameter subgroup of affine conformal automorphisms and is therefore a complete conformal vector field.
Writing $\alpha = a_1+ia_2$, $\beta=b_1+ib_2$ and separating real and imaginary parts, the corresponding real vector field on $\mathbb{R}^2$ is
\[
V
 = a_1\,\partial_x + a_2\,\partial_y
 + b_1\,(x\partial_x + y\partial_y)
 + b_2\,(-y\partial_x + x\partial_y),
\]
which is a linear combination of
\[
\partial_x,\ \partial_y,\ x\partial_x + y\partial_y,\ -y\partial_x + x\partial_y.
\]
These four vector fields are complete and conformal for $\delta$, and hence, by (i), conformal for $g$ as well. They generate all complete conformal vector fields on $(\mathbb{R}^2,\delta)$ and thus on $(\mathbb{R}^2,g)$.
\end{proof}

\begin{remark}
The restriction to \emph{complete} conformal vector fields is essential. 
On $\mathbb{C}$, the vector field $V(z) = z^2 \partial_z$ is holomorphic 
and therefore conformal, but generates the flow $z(t) = z_0/(1-z_0 t)$, 
which has finite maximal existence time for $z_0 \neq 0$. Such fields 
generate an infinite-dimensional Lie algebra.
Moreover, completeness is an \emph{intrinsic} property of the vector 
field (independent of the metric), so the four-dimensional conformal 
algebra is the same for $g$ and $\delta$.
\end{remark}

\subsection{Mixed Killing fields in  Cigar geometry:  Complete characterisation}
In this subsection, we classify strict
mixed Killing fields in a Cigar-type Ricci--Bourguignon metric.
Unless otherwise stated, while working with a mixed Killing field
$V$, we will work on the complement of the set where 
$L_{V} g = 0$ (see note after Corollary \ref{cor:homothetic-mixed-killing}).

%%%%%%%%%%%%%%%%%%%%%%%%%%%%%%%%%%%%%%%%%%%%%%%%%%%%%%%%%%%%
\subsubsection{Angular fields}
%%%%%%%%%%%%%%%%%%%%%%%%%%%%%%%%%%%%%%%%%%%%%%%%%%%%%%%%%%%%

\begin{theorem}[Angular rigidity]\label{thm:Cigar-angular}
Let
\[
V = v(r, \theta)\,\partial_\theta
\]
be a smooth local vector field on a Cigar Ricci--Bourguignon soliton.
Then $V$ is a mixed Killing vector field if and only if $v$ is constant.
In this case, $V$ is a rotational Killing field. 
\end{theorem}
\begin{proof}
Write the Cigar metric   \eqref{Cigar1}   in polar coordinates as
\[
g_{rr}=\frac{1}{E+r^2}=:u(r),\qquad g_{\theta\theta}=u(r)\,r^2,\qquad g_{r\theta}=0.
\]
Let $V=v(r, \theta)\partial_\theta$, so $V^r=0$, $V^\theta=v(r, \theta)$. Using 
$$(L_V g)_{ij}=V^k\partial_k g_{ij}+g_{kj}\partial_i V^k+g_{ik}\partial_j V^k,$$ one computes that the  components of $L_V g$ are:
\[
(L_V g)_{r\theta}=(L_V g)_{\theta r}=g_{\theta\theta}\,\partial_r V^\theta
= u(r)\,r^2   \frac{{\partial}}{{\partial r}} v(r, \theta).
\]
Thus, 
$$  (L_V g)_{r\theta}=(L_V g)_{\theta r}  =  \frac{r^2}{E+r^2}   \frac{{\partial}}{{\partial r}} v(r, \theta).$$
And $(L_V g)_{rr}= 0$,
\[
(L_V g)_{\theta\theta} = \frac{r^2}{E+r^2}   \frac{{\partial}}{{\partial \theta}} v(r, \theta).\]
$$(L_VL_Vg)_{rr}= 2 \frac{r^2}{E+r^2}   (\frac{{\partial}}{{\partial \theta}} v(r, \theta))^2.$$
If $V$ is  mixed-Killing, then 
$$  (L_VL_Vg)_{rr} =  f (L_V g)_{rr}= 0,$$
 This implies that $v(r, \theta) = v(r).$\\
Now, to see that $v(r)$ is constant, we proceed as follows.  
Applying $L_V$ again, one finds that $L_VL_Vg$ has no $r\theta$-component,
as $v(r, \theta)$ does not depend on $\theta$. That is
$(L_VL_Vg)_{r \theta} = 0$, but then mixed Killing condition implies that 
 $(L_V g)_{r\theta} = 0$. Consequently, $v(r, \theta)$ also does not depend 
 on $r$ and thus $v(r, \theta)$ is constant.  \\
 Thus, $V =  C  \frac{{\partial}}{{\partial \theta}}$  is mixed Killing.
 But, then the above considerations show that all the components of  $L_V g$  
 vanish and this gives $V =  C  \frac{{\partial}}{{\partial \theta}}$
 is rotational Killing.
\end{proof}

\begin{remark}
Unlike the product case, angular symmetry exhibits strong rigidity here; mixed Killing implies genuine Killing.
\end{remark}

%%%%%%%%%%%%%%%%%%%%%%%%%%%%%%%%%%%%%%%%%%%%%%%%%%%%%%%%%%%%
\subsubsection{Radial fields}
%%%%%%%%%%%%%%%%%%%%%%%%%%%%%%%%%%%%%%%%%%%%%%%%%%%%%%%%%%%%

\begin{theorem}[Radial mixed Killing fields]\label{thm:radial-Cigar}
Let
\[
V = w(s)\,\partial_s,
\]
be a complete smooth local vector field. 
Then $V$ is a mixed Killing vector field 
if and only if   
\begin{equation}\label{eq:radial-classification}
w(s)^2 = A\,\psi(s)^2 + B,
\end{equation}
for constants $A,B\in\mathbb{R}$.
Moreover:
\begin{enumerate}
\item $V$ is conformal if and only if $B=0$,
\item if $B\neq 0$, then $V$ is mixed Killing but not conformal,
\item  at tip of the Cigar soliton and near zeros $s_{0}$ of  $\psi(s)$,  $V$ 
extends smoothly as  \[
V = w(s)\,\partial_s,
\] where $ w(s)^2 = A\, s^2 + B,$ and $ w(s)^2 = A\, (s - s_{0}) ^2 + B,$ 
respectively.
\end{enumerate}
\end{theorem}

\begin{proof}
Here we will use a Cigar Ricci-Bourguignon metric (\ref{Cigar1})
given by
\begin{equation*}
g_{\mathrm{Cigar-RB}}(t) =  ds^2 + \psi(s)^2\,d\theta^2, \qquad \psi(s)=\tanh s.
\end{equation*}
For the simplicity of computations, here we work in this metric.  In the sequel, we denote by $'$  derivative with respect to $s$  and 
note that  $\psi'$ never vanishes.  
A direct computation yields
\[
L_V g = 2w'\,ds^2 + 2w\psi\psi'\,d\theta^2,
\]
and
\[
L_V L_V g = (2ww'' + 4(w')^2)\,ds^2 + \bigl(2ww'\psi\psi' + 2w^2((\psi')^2 + \psi\psi'')\bigr)\,d\theta^2.
\]
The mixed Killing condition $L_V L_V g = f\,L_V g$ is equivalent to solving the system
\[
\begin{cases}
2ww'' + 4(w')^2 = f \cdot 2w' ,\\
2ww'\psi\psi' + 2w^2((\psi')^2 + \psi\psi'') = f \cdot 2w\psi\psi'.
\end{cases}
\]
Without loss of generality we may assume that $\psi(s) \neq 0$, because  if 
$\psi(s_{0}) = 0$, then from our earlier discussion
(see subsection \ref{C-PC}) We know that  in the neighbourhood
of $s_{0}$ we have $\psi(s) = s - s_{0}$ and $\psi'(s) = 1$. Then we can work on this neighbourhood.\\
 Eliminating $f$ from these equations and simplifying (using  $(\psi\psi')' = (\psi')^2 + \psi\psi''$), we obtain:  
\[
\frac{d}{ds}\log(w'w) = \frac{d}{ds}\log(\psi\psi').
\]
Note that the above expression makes sense, as in the 
complement of $L_{V}g = 0$, $w, w'$ don't vanish.
Integrating,
\[
w'w = C\,\psi\psi',
\]
and hence
\[
\frac{1}{2}(w^2)' = C\,\psi\psi',
\]
 and this again implies that
\[
w^2 = A\,\psi^2 + B
\]
for constants $A,B\in\mathbb{R}$.\\
\noindent 
Conversely, substituting $w^2 = A\psi^2   + B$ into the expressions for $L_V g$ and $L_V L_V g$ shows that the proportionality relation $L_V L_V g = f\,L_V g$ holds on the open  set under consideration. \\ Finally, the conformality condition $L_V g = 2\lambda g$ requires $w' = \lambda$ and $w\psi\psi' = \lambda\psi^2$, which forces $w = c\psi$ and therefore $B = 0$.          
\end{proof}

\begin{corollary}
Let 
\[
V = w(s, \theta)\,\partial_s,
\]
be a complete smooth local vector field on a complete Cigar Ricci--Bourguignon soliton. 
Then $V$ is a mixed Killing vector field  
if and only if $w(s, \theta)$ is a function of  $w(s)$ only.
\end{corollary}
\begin{proof}
Comparing  the Lie derivative components and using the mixed Killing condition, the argument, similar to the proof of the Theorem \ref{thm:Cigar-angular}, implies that 
$w$ does not depend on $\theta$.
  \end{proof}

\begin{corollary}\label{mkc} {(\bf Local characterization of complete mixed Killing vector fields in Cigar metric)}:
Let 
\[
V =   v(r, \theta)\,\partial_\theta +  w(s, \theta)\,\partial_s,
\]
be any complete smooth local vector field on a complete Cigar Ricci--Bourguignon soliton. 
Then $V$ is a mixed Killing vector field  
if and only if  $v(r, \theta) = C,$ constant and 
 $w(s, \theta) = \sqrt{(A\,\psi(s)^2 + B)}$. Thus,
\begin{eqnarray}\label{mkvf}
 V =  C   \partial_{\theta}  +   \sqrt{(A\,\psi(s)^2 + B)}  \partial_s. 
\end{eqnarray}
Moreover:
\begin{enumerate}
\item $V$ is conformal if and only if $B=0$,
\item if $B\neq 0$, then $V$ is mixed Killing but not conformal,
\item   at the tip of the Cigar soliton and near zeros $s_{0}$ of  $\psi(s)$,  $V$ 
extends smoothly as  
$$ V =    C   \partial_{\theta}  +   \sqrt{(A\, (s-s_{0})^2 + B)}  \partial_s,  $$
respectively.
\end{enumerate}
\end{corollary}

\vspace{0.4in}

Let $U_1 = \left((S^1 \setminus (0, 1)) \times (0, \infty)\right)$, $U_2  = \left((S^1 \setminus (0, -1)) \times (0, \infty)\right)$ be the covering of Cigar surface $\Sigma^2 = S^1 \times (0, \infty)$.
From the above corollary, we conclude that:

\begin{theorem}\label{mkc-g} {(\bf Global  characterization of complete mixed Killing vector fields in Cigar metric)}:
Let $V$ be a complete mixed Killing vector field on a Cigar Ricci--Bourguignon soliton $(\Sigma^2 = S^1 \times (0, \infty), g, f) $. 
Then 
\begin{eqnarray}
V & = & \sum_{i=1}^{2} C_i   \partial_{{\theta}_i}  +   \sqrt{(A_i\,\psi(s_{i})^2 + B_i)}  \partial_{s_i}, \\ \nonumber
& = &  V_1 + V_2,
\end{eqnarray}
where $\left(\Omega_1, \partial_{{\theta}_1},  \partial_{s_1} \right)$, $\left(\Omega_2,  \partial_{{\theta}_2}, \partial_{s_2}\right)$ denote the covering by charts of the complement in $\Sigma^2$ of the set $L_{V}g \neq 0$. 
Moreover, 
\begin{enumerate}
\item The constants $C_i, A_i, B_i$ are unique upto diffeomorphisms
of $\Omega_1, \Omega_2$ respectively. And $V_1, V_2$ coincide on
$\Omega_1 \cap \Omega_2$. 
\item $V$ is conformal if and only if $B_i=0$ for at least one 
$i =1,2$.
\item if $B_i\neq 0$ for any one of $i$, then $V$ is mixed Killing but not conformal,
\item   at the tip of the Cigar soliton and near zeros $s_{0}$ of  ${\psi}(s_i)$,  $V$ 
extends smoothly as  
$$ V =\sum_{{i=1}^{2}} C_{i} \partial_{{\theta}_{i}}   +   \sqrt{(A\, (s_{i}-s_{0})^2 + B)}  \partial_{s_{i}}, $$
respectively.
\end{enumerate}
\end{theorem}

\vspace{0.1in}

\begin{remark}
In view of the characterisation of mixed Killing fields obtained here,
it is natural to ask: Can we characterise the zero set 
of mixed Killing fields?  Note that it is known that the zeros 
of closed conformal vector fields are discrete \cite{MR4648543}.
\end{remark}

\vspace{0.2in}

Let $\mathfrak{MK}(g), \mathfrak{conf}(g)$ respectively
denote the vector space of complete mixed Killing vector fields
and complete conformal vector fields, respectively. From Theorem
\ref{theorem:conf-4d}, we know that
the Lie algebra of complete conformal vector fields on $(\mathbb{R}^2,g)$ is $4$--dimensional and is spanned by
\[
\partial_x,\quad \partial_y,\quad -y\partial_x + x\partial_y,\quad x\partial_x + y\partial_y.
\]

\begin{theorem}({\bf Basis for mixed Killing vector fields}):\label{bmkv}
Let $(\mathbb{R}^2,g)$ be a complete Cigar Ricci--Bourguignon soliton. Then $\mathfrak{MK}(g)$ is spanned by 
\[\mathfrak{conf}(g), \; 
 \sqrt{1+ \frac{E}{x^2+y^2} } \; \left(x \partial_x + y \partial_y\right) .
\]
Hence,  $\dim\mathfrak{MK}(g)=5$.
\end{theorem}

\begin{proof}
Recall that from \eqref{mkvf}, any complete mixed Killing vector 
field is locally represented as: 
$$V =  C   \partial_{\theta}  +   \sqrt{(A\,\psi(s)^2 + B)} \partial_s.$$ 
Clearly, note that 
 \[\partial_{\theta} = - y \partial_{x} + x \partial_{y}.\]
And 
 \[\partial_{s}  =  \coth s \; (x \partial_{x} + y \partial_{y})
=  \sqrt{1+ \frac{E}{x^2+y^2} } \; \left(x \partial_x + y \partial_y\right).
\]
Hence, the result follows as stated.
\end{proof}

\section{Geodesic structure and geometric context}\label{geocig}
In this section, we completely characterise the geodesics on a Cigar Ricci-Bourguignon soliton.

\vspace{0.1in}

\noindent
In what follows, it is convenient to work with geodesic polar coordinates $(s,\theta)$ 
as discussed in subsection \ref{C-PC}. Recall that

\begin{equation}\label{eq:geodesic-polar-coord}
r = \sqrt{E}\,\sinh s, \qquad s\in[0,\infty),
\end{equation}
in which the metric takes the warped product form
\begin{equation}\label{eq:warped-form}
g = ds^2 + \psi(s)^2\,d\theta^2, \qquad \psi(s)=\tanh s.
\end{equation}

The metric components are diagonal:
\[
g_{ss} = 1, \qquad g_{\theta\theta} = \psi(s)^2, \qquad g_{s\theta} = 0.
\]

%%%%%%%%%%%%%%%%%%%%%%%%%%%%%%%%%%%%%%%%%%%%%%%%%%%%%%%%%%%%
\subsection{Christoffel symbols and geodesic equations}
%%%%%%%%%%%%%%%%%%%%%%%%%%%%%%%%%%%%%%%%%%%%%%%%%%%%%%%%%%%%

For the warped product metric \eqref{eq:warped-form}, the nonzero Christoffel 
symbols are
\[
\Gamma^s_{\theta\theta} = -\psi(s)\psi'(s),
\qquad
\Gamma^\theta_{s\theta} = \Gamma^\theta_{\theta s} = \frac{\psi'(s)}{\psi(s)}.
\]

A geodesic $\gamma(\sigma) = (s(\sigma), \theta(\sigma))$ satisfies
\begin{equation}\label{eq:geod-s-main}
\ddot{s} - \psi(s)\psi'(s)\dot{\theta}^2 = 0,
\end{equation}
\begin{equation}\label{eq:geod-theta-main}
\ddot{\theta} + 2\frac{\psi'(s)}{\psi(s)}\dot{s}\dot{\theta} = 0,
\end{equation}
where $\sigma$ is an affine parameter and dots denote $d/d\sigma$.

%%%%%%%%%%%%%%%%%%%%%%%%%%%%%%%%%%%%%%%%%%%%%%%%%%%%%%%%%%%%
\subsection{Conservation laws}
%%%%%%%%%%%%%%%%%%%%%%%%%%%%%%%%%%%%%%%%%%%%%%%%%%%%%%%%%%%%

The geodesic equations admit two conserved quantities.

\begin{proposition}[Angular momentum]\label{prop:angular-momentum}
Since the metric is independent of $\theta$, the quantity
\begin{equation}\label{eq:angular-momentum-geo}
\ell := \psi(s)^2\dot{\theta} = \tanh^2(s)\,\dot{\theta}
\end{equation}
is constant along geodesics.
\end{proposition}

\begin{proof}
Equation \eqref{eq:geod-theta-main} can be rewritten as
\[
\frac{d}{d\sigma}\left(\psi^2\dot{\theta}\right) 
= \psi^2\ddot{\theta} + 2\psi\psi'\dot{s}\dot{\theta} = 0.
\]
\end{proof}

\begin{proposition}[Speed parameter]\label{prop:speed}
For an affinely parameterised geodesic, the speed
\begin{equation}\label{eq:speed-geo}
k := \dot{s}^2 +  \psi(s)^2 \dot{\theta}^2 = \dot{s}^2 + \ell^2\coth^2 s
\end{equation}
is constant. For unit-speed geodesics, $k=1$.
\end{proposition}

\begin{proof}
Direct computation using equations \eqref{eq:geod-s-main} and \eqref{eq:geod-theta-main} 
shows that $dk/d\sigma = 0$.
\end{proof}

From these conservation laws, we obtain the radial equation
\begin{equation}\label{eq:radial-equation-s}
\dot{s}^2 = k - \ell^2\coth^2 s.
\end{equation}

%%%%%%%%%%%%%%%%%%%%%%%%%%%%%%%%%%%%%%%%%%%%%%%%%%%%%%%%%%%%
\subsection{Classification of geodesics}
%%%%%%%%%%%%%%%%%%%%%%%%%%%%%%%%%%%%%%%%%%%%%%%%%%%%%%%%%%%%

\begin{theorem}[Geodesic classification]\label{thm:geodesic-classification-main}
Geodesics on the Cigar Ricci-Bourguignon soliton are of two types:

\begin{enumerate}
\item \textbf{Radial geodesics} ($\ell=0$): These satisfy $\theta = \text{const}$ and
\[
s(\sigma) = \pm\sqrt{k}\,\sigma + s_0,
\qquad
r(\sigma) = \sqrt{E}\,\sinh(\pm\sqrt{k}\,\sigma + s_0).
\]
Every radial geodesic passes through the origin $r=0$.

\item \textbf{Non-radial geodesics} ($\ell\neq 0$, with $k>\ell^2$): These satisfy
\begin{equation}\label{eq:s-solution-main}
\cosh s(\sigma)
=
\sqrt{\frac{k}{k-\ell^2}}\,
\cosh\!\bigl(\sqrt{k-\ell^2}\,\sigma\bigr),
\end{equation}
with angular coordinate
\[
\dot{\theta} = \ell\coth^2 s.
\]
These geodesics have a unique turning point at $s_{\min}$ where $\dot{s}=0$, 
given by
\[
\coth^2 s_{\min} = \frac{k}{\ell^2}, 
\qquad
r_{\min} = \sqrt{E}\sinh s_{\min} = \sqrt{\frac{E(k-\ell^2)}{\ell^2}}.
\]
They escape to infinity in both directions: $r(\sigma) \to \infty$ as 
$\sigma \to \pm\infty$.
\end{enumerate}
\end{theorem}

\begin{proof}
For $\ell=0$, equation \eqref{eq:radial-equation-s} gives $\dot{s}^2 = k$, 
yielding the radial case immediately.

For $\ell \neq 0$, using $\coth^2 s = 1 + \csch^2 s$ in \eqref{eq:radial-equation-s} gives
\[
\dot{s}^2 = (k-\ell^2) - \ell^2 \csch^2 s.
\]
The substitution $u = \cosh s$ leads to the standard integral
\[
\int \frac{du}{\sqrt{(k-\ell^2)u^2 - k}} = \pm \sigma + C,
\]
which yields \eqref{eq:s-solution-main} after choosing the origin of $\sigma$ 
appropriately. The turning point corresponds to $\dot{s}=0$.
\end{proof}

\vspace{0.1in}

\noindent
{\bf Key Observations:} The geodesic structure of the Cigar reflects its key geometric features:

\begin{itemize}

\item \textbf{Complete non-compactness:} All geodesics (except those starting at 
the tip) escape to infinity.

\item \textbf{Positive curvature:} The Gaussian curvature $ K = \frac{2E}{(E+r^2)^2}$ is 
everywhere positive but decays exponentially, preventing spiralling behaviour.

\item \textbf{Scattering geometry:} Non-radial geodesics approach a minimum 
radius $r_{\min}$, then recede to infinity, characteristic of scattering in a 
repulsive potential.
\end{itemize}

These properties distinguish the Cigar from both flat space (where geodesics 
are arbitrary straight lines with no special point) and compact positively 
curved spaces like the sphere (where geodesics are closed).\\

\noindent
{\bf Question:}
In view of the results obtained,  here we can ask: 
Do the results obtained here about mixed  Killing vector fields
 in dimension $2$ extend to higher dimensions? More generally, what is the structure of mixed Killing fields? Higher-dimensional gradient Ricci solitons? In particular, a Cigar Ricci-Bourguignon soliton in higher dimensions given by a rotationally symmetric metric:
\[
g_n = ds^2 + \tanh^2(s) \, g_{S^{n-1}},
\]
where $g_{S^{n-1}}$ is the round metric on $S^{n-1}$

\bibliographystyle{alpha}
	\bibliography{refs}
\end{document}